\documentclass[10pt,reqno,oneside]{amsart}

\usepackage[a4paper,margin=1.1in]{geometry}
\usepackage{amsmath,amssymb,amsthm,mathtools,mathrsfs}
\usepackage{enumitem}
\usepackage{booktabs}
\usepackage{microtype}
\usepackage[dvipsnames]{xcolor}
\usepackage{comment}
\usepackage{cases}
\usepackage[numbers]{natbib}

\makeatletter
\renewcommand{\subsection}{\@startsection{subsection}{2}%
  \z@{\linespacing\@plus.7\linespacing}{.5\linespacing}%
  {\normalfont\scshape}}
\makeatother

\newcommand{\R}{\mathbb{R}}
\newcommand{\Z}{\mathbb{Z}}
\DeclareMathOperator{\diag}{diag}
\newcommand{\Sun}{\mathcal{S}}

\theoremstyle{plain}
\newtheorem{theorem}{Theorem}[section]
\newtheorem{proposition}[theorem]{Proposition}
\newtheorem{lemma}[theorem]{Lemma}
\newtheorem{corollary}[theorem]{Corollary}
\newtheorem{conjecture}{Conjecture}

\theoremstyle{definition}
\newtheorem{definition}[theorem]{Definition}

\theoremstyle{remark}

\usepackage{hyperref}
\usepackage[nameinlink]{cleveref}

\hypersetup{
  colorlinks=true,
  linkcolor=blue,
  citecolor=blue,
  urlcolor=blue
}

\title[]{Pendant paths and integral generalized sun graphs}

\author{Rodrigo O. Braga, Jean Carlo Moraes \and Matheus C. Santos}

\thanks{%
\begin{tabular}{@{}l}
Instituto de Matemática e Estatística, Universidade Federal do Rio Grande do Sul, Porto Alegre, Brazil. \\
\end{tabular}}

\date{}

\begin{document}

\begin{abstract}
A graph is integral if the spectrum of its adjacency matrix consists entirely of integers. We prove that every simple graph having a pendant path with at least three edges has an eigenvalue in $(1,2\cos(\pi/9)]$ and one in $[-2\cos(\pi/9),-1)$, and hence is not integral. This settles a conjecture of Braga, Del-Vecchio and Rodrigues (2021) on integral generalized sun graphs. The argument is matrix-theoretic: adjoining a terminal path on three new coordinates to an arbitrary real symmetric matrix produces the same spectral obstruction, and the positive interval above is optimal in this generality. We then disprove a second conjecture of the same authors, which asserts that the cycle of an integral generalized sun graph other than a cycle has length divisible by four. The graph obtained from a hexagon by attaching $6,6,12,6,6$ pendant vertices to five of its six vertices is integral and has $42$ vertices. We show that it is the smallest member of an infinite family governed by the Pell equation $x^{2}-2k^{2}=-7$, and we compute in closed form the characteristic polynomial of the analogous graphs on an arbitrary even cycle. Integrality within this family forces the cycle to be a square or a hexagon, and the square case yields a second infinite family governed by the Pell equation $k^{2}-2c^{2}=1$.

\bigskip


\end{abstract}

\maketitle

\section{Introduction}\label{sec:intro}

We say that a graph is \emph{integral} if the spectrum of its adjacency matrix consists entirely of integers. A connected graph containing exactly one cycle is called \emph{unicyclic}.

The notion of integral graph dates back to Harary and Schwenk \cite{harary}, who asked which graphs have integral spectra and already observed that the problem, in its full generality, appeared to be intractable. The reason became clear much later, when Ahmadi, Alon, Blake and Shparlinski \cite{AHMADI2009547} showed that integral graphs are extremely rare: only a fraction of $2^{-\Omega(n)}$ of the graphs on $n$ vertices has an integral spectrum, a bound subsequently improved to $2^{-cn^{3/2}}$ by Costello and Williams \cite{CostelloWilliams2016}. The natural way around this obstacle has been to restrict attention to structured classes, and in this direction a considerable body of results has accumulated. Complete graphs and complete bipartite graphs $K_{p,q}$ with $pq$ a perfect square are integral, as are the hypercubes and many other Cayley graphs. The connected cubic integral graphs are known to be exactly thirteen, as shown independently by Bussemaker and Cvetkovi\'c
\cite{BussemakerCvetkovic1976} and by Schwenk \cite{Schwenk1978}.  Integral trees have attracted particular attention since the work of Watanabe and Schwenk \cite{WatanabeSchwenk1979}, who determined the integral starlike trees, that is, the integral trees with exactly one vertex of degree larger than two. Brouwer \cite{Brouwer2008} listed all integral trees on at most $50$ vertices, Csikv\'ari \cite{Csikvari2010} constructed integral trees of every even diameter, and Ghorbani, Mohammadian and Tayfeh-Rezaie \cite{GhorbaniMTR2012} settled the odd case, so that integral trees exist with arbitrarily large diameter. We refer the reader to the surveys \cite{Balinska} and \cite{Wang} for these and many other results.

Graphs with few cycles are a natural testing ground, since their structure is close to that of trees while their spectra are already substantially harder to control. For unicyclic graphs, Omidi \cite{Omidi2009} established necessary conditions for integrality; these conditions, however, did not lead to any integral unicyclic graph other than the cycles $C_3$, $C_4$ and $C_6$. It is worth stressing how special this situation is: if one replaces the adjacency matrix by the Laplacian, the signless Laplacian or the normalized Laplacian, the corresponding integral unicyclic graphs have all been completely determined \cite{Liu2010,Zhang2017,vanDam2011}, whereas for the adjacency matrix the problem remains open. A further source of information is the eigenvalue location algorithm of Braga, Rodrigues and Trevisan \cite{BragaTrevisan2017}, which computes in linear time the number of eigenvalues of a unicyclic graph in a given real interval; a computer search based on it produced only three integral unicyclic graphs on at most $21$ vertices besides the three cycles above.

The first integral unicyclic graphs that are not cycles were obtained in \cite{BRAGA2021281}, where three infinite families were constructed from distinct particular solutions of a Diophantine equation. The unicyclic graphs considered in that work are obtained by attaching pendant paths to the vertices of a cycle and are referred to as \emph{generalized sun graphs}. Beyond establishing necessary conditions for certain generalized sun graphs to be integral, the authors of \cite{BRAGA2021281} proposed the following two conjectures.

\begin{conjecture}\label{conj:one}
   There is no integral generalized sun graph where a path $P_t$, with $t \geq 3$, is attached to a cycle vertex.
\end{conjecture}

\begin{conjecture}\label{conj:two}
    If a generalized sun graph that is not a cycle is integral, then the order of the generalized sun graph's cycle is a multiple of 4.
\end{conjecture}

In this paper we prove Conjecture~\ref{conj:one} and we show that Conjecture~\ref{conj:two} is false, by exhibiting an infinite family of integral generalized sun graphs whose cycle has order six.

Let us fix the notation, which is the same used in \cite{BRAGA2021281}. We write $P_t$ for the path on $t$ vertices. For a graph $H$ and a vertex $v$ of $H$, \emph{attaching} $P_t$ to $v$ means adding $t$ new vertices $u_1,\ldots,u_t$ together with the edges $vu_1$ and $u_iu_{i+1}$ for $1\leq i\leq t-1$. The path $vu_1\cdots u_t$ is then called a \emph{pendant path of length $t$ at $v$}. More generally, a path $vu_1\cdots u_t$ in a graph $G$ is a pendant path of length $t$ at $v$ if $\deg_G(u_i)=2$ for $1\leq i\leq t-1$ and $\deg_G(u_t)=1$; equivalently, $G$ is obtained from $G-\{u_1,\ldots,u_t\}$ by attaching $P_t$ to $v$. In particular, a pendant vertex is a pendant path of length $1$. We denote by $C_b(n_1P_{t_1},\ldots,n_bP_{t_b})$ the generalized sun graph obtained from the cycle $C_b=v_1v_2\cdots v_bv_1$ by attaching $n_k$ copies of the path $P_{t_k}$ to the cycle vertex $v_k$, for $1\leq k\leq b$, and we abbreviate this to $C_{b,t}(n_1,\ldots,n_b)$ when all the attached paths have the same length $t$. In particular, $C_{b,1}(n_1,\ldots,n_b)$ is the cycle $C_b$ with $n_k$ pendant vertices attached to $v_k$.

The two conjectures are treated by completely different methods, which is reflected in the organization of the paper.

Section~\ref{sec:conj1} is devoted to Conjecture~\ref{conj:one}. Let $\rho=2\cos(\pi/9)$. Theorem~\ref{thm:long-pendant-paths} proves, in a considerably stronger form, that every simple graph carrying a pendant path with at least three edges has an eigenvalue in $(1,\rho]$ and one in $[-\rho,-1)$, and hence cannot be integral. In particular, it recovers the fact, due to Watanabe and Schwenk \cite{WatanabeSchwenk1979}, that every arm of an integral starlike tree has length at most two. The proof compares the inertia of two shifted matrices by means of Schur complements and Sylvester's law of inertia. It uses no property of the part of the adjacency matrix outside the last three vertices except real symmetry. We therefore obtain the same conclusion for an arbitrary real symmetric matrix to which a terminal path on three coordinates is adjoined, and Proposition~\ref{prop:sharp-interval} shows that $(1,\rho]$ is optimal in that matrix setting.

Section~\ref{sec:pendant} deals with Conjecture~\ref{conj:two}. Since Theorem~\ref{thm:long-pendant-paths} leaves only pendant paths with one or two edges available, we look for counterexamples among the graphs $C_{b,1}(n_1,\ldots,n_b)$, in which only pendant vertices are attached. For these graphs the characteristic polynomial is governed by a $b\times b$ quadratic matrix polynomial whose size does not grow with the number of attached vertices. This reduction, established in Lemma~\ref{lem:reduction}, is the tool used throughout. It yields at once the counterexample of Section~\ref{sec:counterexample}: the graph $C_{6,1}(0,6,6,12,6,6)$, on $42$ vertices, is integral, and its cycle is a hexagon. We then ask what lies behind this example. In Section~\ref{sec:hexagon} we show that it is the first member of an infinite family $C_{6,1}(0,a,a,2a,a,a)$, which is integral exactly when $a=m(m+1)$ and $2m^2+2m+4$ is a perfect square; in the variable $x=2m+1$ the latter condition is the generalized Pell equation $x^2-2k^2=-7$, which has infinitely many solutions, so that Conjecture~\ref{conj:two} fails for infinitely many graphs.

The shape of that family makes sense on any even cycle, and Section~\ref{sec:general} studies the resulting graphs $\Sun_{2h}(a)$, in which two antipodal cycle vertices receive respectively no pendant vertex and $2a$ of them, while each of the remaining $2h-2$ cycle vertices receives $a$. Theorem~\ref{thm:general-char-poly} computes their characteristic polynomial in closed form for every $h$, as an explicit quadratic factor times the square of a Chebyshev-like polynomial. Theorem~\ref{thm:only-4-and-6} then shows directly from this factorization that $\Sun_{2h}(a)$ can be integral only for $h\in\{2,3\}$. Thus the hexagon of Section~\ref{sec:hexagon} and the square are the only cycle lengths available to this shape.

For the square, Theorem~\ref{thm:general-char-poly} yields the family
$\Sun_4(a)=C_{4,1}(0,a,2a,a)$. Veras \cite{veras} found its $68$-vertex member
$C_{4,1}(0,16,32,16)$ by computer search and stated that the nonzero
eigenvalues of $\Sun_4(a)$ are $\pm\sqrt a$ and $\pm\sqrt{2a+4}$, so that
$\Sun_4(a)$ is integral when $a$ and $2a+4$ are perfect squares. Here this
family is obtained as the case $h=2$ of Theorem~\ref{thm:general-char-poly},
with its full spectrum, and its integral members are described completely
through the classical Pell equation $r^2-2s^2=1$. This complements the isolated $13$-vertex graph with pendant vertices, $C_{4,1}(6,0,3,0)$, already known from \cite[Corollary~4.3]{BRAGA2021281}.

Section~\ref{sec:search} reports an exhaustive computer search over explicitly stated boxes of parameters. Within those boxes, the $42$-vertex graph is the only example whose cycle length is not divisible by four. Finally, Section~\ref{sec:conclusion} presents our concluding remarks and discusses possible directions for future research.

\section{A sharp spectral obstruction for long pendant paths}\label{sec:conj1}

Throughout this section, for a real symmetric matrix $A$ we denote by
\(
n_+(A),\, n_-(A),\, n_0(A)
\)
the numbers of positive, negative, and zero eigenvalues of $A$, respectively, counted with multiplicity. We write $\sigma(A)$ for its spectrum. The spectral theorem gives
\begin{equation}\label{eq:counting-identities}
n_+(A-\lambda I)=\#\{\mu\in\sigma(A):\mu>\lambda\},
\qquad
n_+(A-\lambda I)+n_0(A-\lambda I)=\#\{\mu\in\sigma(A):\mu\geq\lambda\}.
\end{equation}
Set
\(
\rho:=2\cos\frac{\pi}{9}=1.879385\ldots,
\)
the largest root of $x^3-3x-1$. The following result settles Conjecture~\ref{conj:one} and applies to arbitrary simple graphs, rather than only to generalized sun graphs.

\begin{theorem}\label{thm:long-pendant-paths}
Let $G$ be a simple graph having a pendant path of length $t\geq3$. Then $A(G)$ has an eigenvalue in $(1,\rho]$ and an eigenvalue in $[-\rho,-1)$, where $\rho=2\cos(\pi/9)$. In particular, $G$ is not integral. Consequently, Conjecture~\ref{conj:one} holds.
\end{theorem}

\begin{proof}
Let $v_0v_1\cdots v_t$ be a pendant path of length $t\geq3$ in $G$, put $v=v_{t-3}$ and $H=G-\{v_{t-2},v_{t-1},v_t\}$. Then $vv_{t-2}v_{t-1}v_t$ is a pendant path of length $3$ at $v$, that is, $G$ is obtained from $H$ by attaching $P_3$ to $v$. More generally, let $M\in\R^{n\times n}$ be an arbitrary real symmetric matrix, let $e_v$ be a standard coordinate vector, and consider

$$
\widehat M_v=
\begin{pmatrix}
M & e_v & 0 & 0\\
e_v^{\mathsf T} & 0 & 1 & 0\\
0 & 1 & 0 & 1\\
0 & 0 & 1 & 0
\end{pmatrix}.
$$

Note that, after a suitable ordering of the vertices, $A(G)=\widehat M_v$ with $M=A(H)$. We prove the stronger matrix statement
\begin{equation}\label{eq:sharp-matrix-obstruction}
\sigma(\widehat M_v)\cap(1,\rho]\neq\varnothing.
\end{equation}

Let

$$
B=
\begin{pmatrix}
0&1&0\\
1&0&1\\
0&1&0
\end{pmatrix},
\qquad
f=
\begin{pmatrix}
1\\0\\0
\end{pmatrix}.
$$

Then

$$
\widehat M_v=
\begin{pmatrix}
M&e_vf^{\mathsf T}\\
fe_v^{\mathsf T}&B
\end{pmatrix},
$$

\noindent and $\sigma(B)=\{0,\pm\sqrt2\}$. For $\lambda\notin\{0,\pm\sqrt2\}$, the matrix $B-\lambda I_3$ is invertible. Writing $C=e_vf^{\mathsf T}$, block Gaussian elimination on rows and the corresponding columns shows that $\widehat M_v-\lambda I$ is congruent to

$$
S_\lambda\oplus(B-\lambda I_3),
\qquad
S_\lambda=M-\lambda I_n-C(B-\lambda I_3)^{-1}C^{\mathsf T}.
$$

Since $C=e_vf^{\mathsf T}$, this becomes $S_\lambda=M-\lambda I_n-q(\lambda)e_ve_v^{\mathsf T}$, where
\begin{equation}\label{eq:q-explicit}
q(\lambda)
=f^{\mathsf T}(B-\lambda I_3)^{-1}f=
\frac{
\det\begin{pmatrix}
-\lambda&1\\
1&-\lambda
\end{pmatrix}}
{\det(B-\lambda I_3)}=
\frac{\lambda^2-1}{\lambda(2-\lambda^2)}.
\end{equation}
Sylvester's law of inertia therefore gives
\begin{equation}\label{eq:Sylvester}
n_+(\widehat M_v-\lambda I)=
n_+(S_\lambda)+n_+(B-\lambda I_3),
\end{equation}
together with the analogous identity for $n_-$ and $n_0$.

At $\lambda=1$, we have $q(1)=0$ and hence $S_1=M-I_n$. Moreover, $\sigma(B-I_3)=\{\sqrt2-1,-1,-\sqrt2-1\}$, so $B-I_3$ has exactly one positive eigenvalue. Thus
\begin{equation}\label{eq:level-one}
n_+(\widehat M_v-I)=n_+(S_1)+1.
\end{equation}

Now fix $r\in(\rho,2)$. Since $r>\rho>\sqrt2$, the three eigenvalues $\sqrt2-r$, $-r$, and $-\sqrt2-r$ of $B-rI_3$ are negative. Hence
\begin{equation}\label{eq:level-r}
n_+(\widehat M_v-rI)=n_+(S_r),
\qquad
n_0(\widehat M_v-rI)=n_0(S_r).
\end{equation}

We also have $S_1-S_r=(r-1)I_n+q(r)e_ve_v^{\mathsf T}$. Since $e_ve_v^{\mathsf T}$ vanishes on $e_v^\perp$, this matrix has eigenvalue $r-1>0$ on $e_v^\perp$. Its eigenvalue in the direction of $e_v$ is
\begin{equation}\label{eq:sharp-factorization}
(r-1)+q(r)=
\frac{(r-1)(r^3-3r-1)}{r(r^2-2)}>0,
\end{equation}
because $r>\rho$, where $\rho$ is the largest root of $x^3-3x-1$, and $r>\sqrt2$. Therefore $S_1-S_r$ is positive definite.

Let $W=E_+(S_r)\oplus E_0(S_r)$ be the direct sum of the eigenspaces of $S_r$ corresponding to its nonnegative eigenvalues. For every nonzero $x\in W$, we have $x^{\mathsf T}S_rx\geq0$ and $x^{\mathsf T}(S_1-S_r)x>0$, and hence $x^{\mathsf T}S_1x>0$. Thus $S_1$ is positive definite on $W$. Since the largest possible dimension of a subspace on which the quadratic form of $S_1$ is positive definite is $n_+(S_1)$, it follows that
\begin{equation}\label{eq:inertia-comparison}
n_+(S_r)+n_0(S_r)=\dim W\leq n_+(S_1).
\end{equation}

For a symmetric matrix $X$ and an interval $J\subset\R$, let $N_J(X)$ denote the number of eigenvalues of $X$ lying in $J$, counted with multiplicity. Using \eqref{eq:counting-identities}, \eqref{eq:level-one}, \eqref{eq:level-r}, and \eqref{eq:inertia-comparison}, we obtain
\begin{align*}
N_{(1,r)}(\widehat M_v)
&=n_+(\widehat M_v-I)
-n_+(\widehat M_v-rI)
-n_0(\widehat M_v-rI)\\
&=1+n_+(S_1)-n_+(S_r)-n_0(S_r)\geq1.
\end{align*}
This holds for every $r\in(\rho,2)$, so $\widehat M_v$ has an eigenvalue greater than $1$. Let $\mu=\min\bigl(\sigma(\widehat M_v)\cap(1,\infty)\bigr)$. If $\mu>\rho$, then we may choose $r$ such that $\rho<r<\min\{\mu,2\}$. Since $1<r<\mu$ and $\mu$ is the smallest eigenvalue greater than $1$, this would give $N_{(1,r)}(\widehat M_v)=0$, contradicting the preceding inequality. Therefore $\mu\leq\rho$, which proves \eqref{eq:sharp-matrix-obstruction}.

To obtain the negative interval, let $D=\diag(I_n,-1,1,-1)$. Then

$$
D\widehat M_vD=-\widehat{(-M)}_v.
$$

Applying \eqref{eq:sharp-matrix-obstruction} to $-M$ and using the similarity of $\widehat M_v$ and $D\widehat M_vD$ gives $\sigma(\widehat M_v)\cap[-\rho,-1)\neq\varnothing$.

Returning to $M=A(H)$ proves both spectral assertions for $G$. Since neither
interval contains an integer, $G$ is not integral. Finally, attaching $P_t$ to
a cycle vertex creates a pendant path of length $t$, so no generalized sun
graph with an attached copy of $P_t$, $t\geq3$, is integral. Hence, this proves Conjecture~\ref{conj:one}. 
\end{proof}

Theorem~\ref{thm:long-pendant-paths} contains, as a special case, part of the classification of integral starlike trees due to Watanabe and Schwenk \cite{WatanabeSchwenk1979}: an arm of length at least three of a starlike tree is a pendant path of length at least three at its central vertex, so every arm of an integral starlike tree has length at most two. For trees with two vertices of degree larger than two, Watanabe and Schwenk treated the case in which these vertices are adjacent, and Brouwer, Del-Vecchio, Jacobs, Trevisan and Vinagre \cite{BrouwerDJTV2011} showed that no such tree is integral when they are nonadjacent. These results concern trees of a prescribed shape, whereas Theorem~\ref{thm:long-pendant-paths} imposes no condition on the graph outside the pendant path.

The constant $\rho$ and the choice of a closed right endpoint are forced in the general matrix statement used in the proof.

\begin{proposition}\label{prop:sharp-interval}
The interval $(1,\rho]$ in \eqref{eq:sharp-matrix-obstruction} is optimal among intervals contained in $(1,\rho]$, that is, no proper interval of the form $(a,b]\subsetneq(1,\rho]$ contains an eigenvalue of $\widehat M_v$ for every real symmetric matrix $M$ and every coordinate $v$.
\end{proposition}

\begin{proof}
First take $M=[1]$. Then
\[
\det(\lambda I-\widehat M)
=(\lambda-1)(\lambda^3-3\lambda-1),
\]
and hence
\[
\sigma(\widehat M)
=\left\{2\cos\frac{7\pi}{9},\ 2\cos\frac{5\pi}{9},\ 1,\ \rho\right\}.
\]
Thus $\rho$ is the only eigenvalue in $(1,\rho]$, so the upper endpoint cannot be decreased or omitted.

It remains to show that the lower endpoint cannot be increased. Replace $M=[1]$ by $M_\varepsilon=[1+\varepsilon]$, and write $\widehat M_\varepsilon$ for the corresponding extended matrix. Since the eigenvalue $1$ of $\widehat M_0$ is simple, it determines an analytic eigenvalue branch $\lambda_1(\varepsilon)$ satisfying $\lambda_1(0)=1$. An associated eigenvector at $\varepsilon=0$ is $u=(1,0,-1,-1)^{\mathsf T}$, while $\widehat M_\varepsilon'=\diag(1,0,0,0)$. The standard first-order perturbation formula therefore gives

$$
\lambda_1'(0)
=\frac{u^{\mathsf T}\widehat M_\varepsilon' u}{u^{\mathsf T}u}
=\frac13,
\qquad
\lambda_1(\varepsilon)
=1+\frac{\varepsilon}{3}+O(\varepsilon^2).
$$

The simple eigenvalue $\rho$ likewise determines an analytic branch $\lambda_\rho(\varepsilon)$ with $\lambda_\rho(0)=\rho$. Its derivative at $0$ is positive, because the same perturbation formula applies and every eigenvector associated with $\rho$ has a nonzero first coordinate. Hence $\lambda_\rho(\varepsilon)>\rho$ for all sufficiently small $\varepsilon>0$, while the other two eigenvalues remain negative by continuity. Consequently, for every $a>1$, choosing $\varepsilon>0$ sufficiently small gives $1<\lambda_1(\varepsilon)<a$ and leaves no eigenvalue in $(a,\rho]$. Thus no interval $(a,\rho]$ with $a>1$ has the universal property asserted in the proposition.  
\end{proof}

We note that the proof of Theorem~\ref{thm:long-pendant-paths} used no combinatorial property of $M$. It therefore also applies to weighted and signed graphs, and to generalized adjacency matrices with an arbitrary diagonal on the original vertices, provided that the three new vertices have zero diagonal and form a unit-weight terminal path joined to the original matrix by one unit entry. The bound is sharp for this matrix class; Proposition~\ref{prop:sharp-interval} does not assert sharpness within the narrower class of adjacency matrices of simple graphs.

\section{Integral generalized sun graphs with pendant vertices}\label{sec:pendant}

We now turn to Conjecture~\ref{conj:two}. By Theorem~\ref{thm:long-pendant-paths}, an integral generalized sun graph can only carry pendant paths of length one or two, that is, copies of $P_1$ and $P_2$. The evidence for the conjecture came primarily from the second possibility: it was proved in \cite[Theorem 3.1]{BRAGA2021281} that if copies of $P_2$ are attached to exactly two cycle vertices and the graph is integral, then the cycle has length four and the two vertices are adjacent, and the three infinite families constructed there are all of the form $C_{4,2}(p,q,0,0)$. The sporadic examples found by computer search in \cite{BragaTrevisan2017} also have
a square as their cycle, and so do the integral graphs with pendant vertices only that were previously known, namely $C_{4,1}(6,0,3,0)$ \cite[Corollary~4.3]{BRAGA2021281} and the graphs $C_{4,1}(0,a,2a,a)$ reported by Veras \cite{veras}. This accumulation of evidence at $b=4$ made Conjecture~\ref{conj:two} natural.

Graphs carrying only pendant vertices have received comparatively little
attention, and the examples just mentioned do not contradict
Conjecture~\ref{conj:two}, since their cycle is a square. This motivates us to
look for counterexamples among the graphs $C_{b,1}(n_1,\dots,n_b)$ on longer
cycles.

We begin with the elementary reduction that will be used throughout: for graphs of the form $C_{b,1}(n_1,\dots,n_b)$ the characteristic polynomial is governed by a $b\times b$ matrix polynomial whose size does not grow with the number of attached pendant vertices.

\begin{lemma}\label{lem:reduction}
Let $G=C_{b,1}(n_1,\dots,n_b)$, let $n=b+\sum_{k=1}^{b}n_k$ be its order, let $C=A(C_b)$ and let $D=\diag(n_1,\dots,n_b)$. Then
\begin{equation}\label{eq:reduction}
p_G(\lambda)=\det\bigl(\lambda I_n-A(G)\bigr)
=\lambda^{\,n-2b}\,\det\bigl(\lambda^2I_b-\lambda C-D\bigr).
\end{equation}
\end{lemma}

\begin{proof}
Order the $b$ cycle vertices first and the $n-b$ pendant vertices last. Writing $R\in\R^{b\times(n-b)}$ for the cycle-to-pendant incidence matrix, we have
\[
A(G)=
\begin{pmatrix}
C & R\\
R^{\mathsf T} & 0
\end{pmatrix}.
\]
Since each pendant vertex is adjacent to exactly one cycle vertex, the rows of $R$ have pairwise disjoint supports, and the $k$-th row has exactly $n_k$ nonzero entries; hence $RR^{\mathsf T}=D$. For $\lambda\neq0$ the block $\lambda I_{n-b}$ is invertible, and taking its Schur complement gives
\[
p_G(\lambda)
=\lambda^{\,n-b}\det\bigl(\lambda I_b-C-\lambda^{-1}RR^{\mathsf T}\bigr)
=\lambda^{\,n-b}\det\bigl(\lambda I_b-C-\lambda^{-1}D\bigr)
=\lambda^{\,n-2b}\det\bigl(\lambda^2I_b-\lambda C-D\bigr).
\]
The left-hand side is a polynomial. The identity for $\lambda\neq0$ shows that the apparent negative power on the right, which can occur when $n<2b$, is removable: the determinant contains the required power of $\lambda$. After this cancellation, the right-hand side is a polynomial, and the identity holds for every $\lambda$.  
\end{proof}

\subsection{A counterexample to Conjecture~\ref{conj:two}}\label{sec:counterexample}

Lemma~\ref{lem:reduction} reduces the integrality of a graph on $42$ vertices to the factorization of a $6\times6$ determinant, and this is enough to disprove Conjecture~\ref{conj:two} at once.

\begin{proposition}\label{prop:counterexample}
The generalized sun graph
\[
G=C_{6,1}(0,6,6,12,6,6)
\]
on $42$ vertices is integral, with
\[
\sigma(G)=\bigl\{0^{[32]},\ \pm2^{[2]},\ \pm3^{[2]},\ \pm4\bigr\}.
\]
In particular, $G$ is an integral generalized sun graph which is not a cycle, whereas its unique cycle has order $6$. Consequently, Conjecture~\ref{conj:two} is false.
\end{proposition}

\begin{proof}
Here $b=6$, $n=42$ and $D=\diag(0,6,6,12,6,6)$, so that Lemma~\ref{lem:reduction} gives $p_G(\lambda)=\lambda^{30}\det Q$ with
\[
Q=\lambda^2I_6-\lambda C-D,
\qquad
C=A(C_6).
\]
Writing $z=\lambda^2$ and expanding, one finds
\[
\det Q=z\,\bigl[(z-6)^2-z\bigr]^2\,(z-16)
=\lambda^2(\lambda^2-4)^2(\lambda^2-9)^2(\lambda^2-16),
\]
where we used $(z-6)^2-z=z^2-13z+36=(z-4)(z-9)$. Hence
\[
p_G(\lambda)=\lambda^{32}(\lambda^2-4)^2(\lambda^2-9)^2(\lambda^2-16),
\]
and all the eigenvalues of $G$ are integers, with the multiplicities stated. The graph has a unique cycle, of order $6$, and it is not a cycle itself.  
\end{proof}

A detailed computation of $\det Q$, for a general parameter in place of $6$, is carried out in the proof of Theorem~\ref{thm:char-poly-hexagon} below.

It is instructive to see why such a counterexample could not be detected by the methods of \cite{BRAGA2021281}. The evidence supporting Conjecture~\ref{conj:two} there comes from generalized sun graphs whose pendant paths are copies of $P_2$. For those graphs, running the eigenvalue location algorithm of \cite{BragaTrevisan2017} with $\alpha=0$ detaches every copy of $P_2$ from the cycle, and shows that $0$ is an eigenvalue of the graph only when it is an eigenvalue of the bare cycle $C_b$, which forces $b\equiv0\pmod4$; this is the first part of the proof of \cite[Theorem~3.1]{BRAGA2021281}. When pendant vertices are attached instead, the mechanism is entirely different: two pendant vertices at a common cycle vertex already produce the eigenvalue $0$, since their difference is an eigenvector, the cycle vertex is then annihilated by the algorithm, and no constraint on $b$ survives. This is precisely what happens in the graph above, which has $32$ zero eigenvalues although $C_6$ has none. We also note that $42$ vertices is well beyond the reach of the exhaustive search up to $21$ vertices reported in \cite{BRAGA2021281}.

\subsection{An infinite family with a hexagonal cycle}\label{sec:hexagon}

The graph of Proposition~\ref{prop:counterexample} is not isolated. Replacing the numbers $6$ and $12$ by $a$ and $2a$ we obtain a one-parameter family
\[
\Sun_6(a):=C_{6,1}(0,a,a,2a,a,a),
\qquad a\geq1,
\]
of bipartite unicyclic graphs on $6a+6$ vertices, whose characteristic polynomial can be computed in closed form. The computation uses the reflection symmetry of the hexagon, and it is the model for the more general one carried out in Section~\ref{sec:general}.

\begin{theorem}\label{thm:char-poly-hexagon}
For every integer $a\geq1$,
\begin{equation}\label{eq:char-poly-hexagon}
p_{\Sun_6(a)}(\lambda)
=\lambda^{6a-4}\,(\lambda^2-\lambda-a)^2\,(\lambda^2+\lambda-a)^2\,(\lambda^2-2a-4).
\end{equation}
\end{theorem}

\begin{proof}
By Lemma~\ref{lem:reduction} with $b=6$ and $n=6a+6$ we have $p_{\Sun_6(a)}(\lambda)=\lambda^{6a-6}\det Q$, where $Q=\lambda^2I_6-\lambda C-D$, $C=A(C_6)$ and $D=\diag(0,a,a,2a,a,a)$. Put $z=\lambda^2$.

The reflection of the hexagon fixing $v_1$ and $v_4$ and exchanging $v_2\leftrightarrow v_6$ and $v_3\leftrightarrow v_5$ preserves both $C$ and $D$. Hence $Q$ is block diagonalized by the orthonormal basis of $\R^6$ consisting of the antisymmetric vectors
\[
\frac{e_2-e_6}{\sqrt2},\qquad\frac{e_3-e_5}{\sqrt2},
\]
together with the symmetric vectors
\[
e_1,\qquad\frac{e_2+e_6}{\sqrt2},\qquad\frac{e_3+e_5}{\sqrt2},\qquad e_4 .
\]
On the antisymmetric subspace, $Q$ is represented by
\[
K_-=\begin{pmatrix}z-a&-\lambda\\-\lambda&z-a\end{pmatrix},
\qquad
\det K_-=(z-a)^2-z ,
\]
while on the symmetric subspace, using
\[
\Bigl\langle e_1,\,C\tfrac{e_2+e_6}{\sqrt2}\Bigr\rangle=\sqrt2,
\qquad
\Bigl\langle \tfrac{e_2+e_6}{\sqrt2},\,C\tfrac{e_3+e_5}{\sqrt2}\Bigr\rangle=1,
\qquad
\Bigl\langle \tfrac{e_3+e_5}{\sqrt2},\,Ce_4\Bigr\rangle=\sqrt2,
\]
it is represented by the tridiagonal matrix
\[
K_+=
\begin{pmatrix}
z & -\sqrt2\lambda & 0 & 0\\
-\sqrt2\lambda & z-a & -\lambda & 0\\
0 & -\lambda & z-a & -\sqrt2\lambda\\
0 & 0 & -\sqrt2\lambda & z-2a
\end{pmatrix}.
\]
Expanding along the last row, the leading principal minors $\Delta_j$ of $K_+$ satisfy
\begin{align*}
\Delta_1&=z,\\
\Delta_2&=(z-a)\Delta_1-2z=z(z-a-2),\\
\Delta_3&=(z-a)\Delta_2-z\Delta_1=z\bigl[(z-a)(z-a-2)-z\bigr],\\
\Delta_4&=(z-2a)\Delta_3-2z\Delta_2 .
\end{align*}
Set $w:=(z-a)^2-z=\det K_-$. Since $(z-a)(z-a-2)-z=w-2(z-a)$, the last line gives
\begin{align*}
\Delta_4/z
&=(z-2a)\bigl(w-2(z-a)\bigr)-2z(z-a-2)\\
&=(z-2a)w-2(z-a)(z-2a)-2z(z-a)+4z\\
&=(z-2a)w-2(z-a)(2z-2a)+4z\\
&=(z-2a)w-4\bigl[(z-a)^2-z\bigr]
=(z-2a-4)\,w .
\end{align*}
Hence $\det K_+=z\,w\,(z-2a-4)$, and therefore
\[
\det Q=\det K_-\cdot\det K_+=z\,w^2\,(z-2a-4)
=\lambda^2\bigl[(\lambda^2-a)^2-\lambda^2\bigr]^2(\lambda^2-2a-4).
\]
Finally $(\lambda^2-a)^2-\lambda^2=(\lambda^2-\lambda-a)(\lambda^2+\lambda-a)$, and multiplying by $\lambda^{6a-6}$ gives \eqref{eq:char-poly-hexagon}.  
\end{proof}

Reading integrality off \eqref{eq:char-poly-hexagon} is now a matter of deciding when three quadratic polynomials have integer roots.

\begin{theorem}\label{thm:integrality-criterion}
Let $a\geq1$ be an integer. The generalized sun graph $\Sun_6(a)=C_{6,1}(0,a,a,2a,a,a)$ is integral if and only if there exist integers $m\geq1$ and $k\geq1$ such that
\begin{equation}\label{eq:integrality-conditions}
a=m(m+1)
\qquad\text{and}\qquad
2m^2+2m+4=k^2 .
\end{equation}
In this case $\Sun_6(a)$ has $6m(m+1)+6$ vertices and
\[
\sigma\bigl(\Sun_6(a)\bigr)=\Bigl\{0^{[6a-4]},\ \pm m^{[2]},\ \pm(m+1)^{[2]},\ \pm k\Bigr\}.
\]
\end{theorem}

\begin{proof}
By Theorem~\ref{thm:char-poly-hexagon}, the nonzero eigenvalues of $\Sun_6(a)$ are the roots of $\lambda^2-\lambda-a$ and of $\lambda^2+\lambda-a$, each with multiplicity two, together with the two roots of $\lambda^2-2a-4$.

The roots of $\lambda^2-\lambda-a$ are $\tfrac12\bigl(1\pm\sqrt{1+4a}\bigr)$. They are integers if and only if $1+4a$ is a perfect square, necessarily odd, say $1+4a=(2m+1)^2$ with $m\geq1$; this is equivalent to $a=m(m+1)$, and in that case the roots are $m+1$ and $-m$. Replacing $\lambda$ by $-\lambda$, the roots of $\lambda^2+\lambda-a$ are then $m$ and $-(m+1)$.

The remaining roots are $\pm\sqrt{2a+4}$, which are integers if and only if $2a+4=2m^2+2m+4$ is a perfect square, say $k^2$ with $k\geq1$. This establishes \eqref{eq:integrality-conditions}. The multiplicity $6a-4$ of the eigenvalue $0$ and the multiplicities of the remaining eigenvalues are read off from \eqref{eq:char-poly-hexagon}, and the order is $6+6a=6m(m+1)+6$.  
\end{proof}

The second condition in \eqref{eq:integrality-conditions} is a Pell equation in disguise. Indeed, substituting $x=2m+1$ in $2m^2+2m+4=k^2$ gives $x^2+7=2k^2$, that is,
\begin{equation}\label{eq:pell}
x^2-2k^2=-7 .
\end{equation}
The solutions of \eqref{eq:pell} are permuted by multiplication by the fundamental unit of norm $1$ of $\Z[\sqrt2]$, which acts here as the map $(x,k)\mapsto(3x+4k,\,2x+3k)$: if $(x,k)$ solves \eqref{eq:pell}, then
\[
(3x+4k)^2-2(2x+3k)^2
=(9-8)x^2+(24-24)xk+(16-18)k^2
=x^2-2k^2 ,
\]
so the image is again a solution. Note that this map preserves the parities of $x$ and $k$, so that the constraint that $x$ be odd is automatically maintained. The next lemma shows that this action accounts for all solutions; the descent argument we give is elementary and self-contained, although the statement can also be deduced from the general theory of norm equations in real quadratic fields, for which we refer to \cite[Chapter~17]{IrelandRosen}.

\begin{lemma}\label{lem:pell}
The positive solutions of \eqref{eq:pell} form two infinite orbits of the map $(x,k)\mapsto(3x+4k,\,2x+3k)$, with seeds $(1,2)$ and $(5,4)$. In every solution $x$ is odd and $k$ is even, and in terms of $m=(x-1)/2$ the admissible values are
\[
m=0,\ 2,\ 5,\ 15,\ 32,\ 90,\ 189,\ 527,\ 1104,\ 3074,\ 6437,\ 17919,\ 37520,\ \dots
\]
\end{lemma}

\begin{proof}
If $(x,k)$ solves \eqref{eq:pell}, then $x^2=2k^2-7$ is odd, so $x$ is odd; consequently $x^2\equiv1\pmod 8$, whence $2k^2=x^2+7\equiv0\pmod 8$ and $k$ is even. That the two announced orbits consist of solutions follows from the invariance computed above, together with $1^2-2\cdot2^2=5^2-2\cdot4^2=-7$.

For the converse we use descent. It is convenient to allow $x\in\Z$ of either sign for a moment, keeping $k>0$; note that $(x,k)$ solves \eqref{eq:pell} if and only if $(-x,k)$ does. Given such a solution, put
\[
(x',k')=(3x-4k,\,3k-2x),
\]
which again satisfies $x'^2-2k'^2=x^2-2k^2=-7$ by the same computation. Moreover $(x,k)=(3x'+4k',\,2x'+3k')$, so the two maps are mutually inverse.

From $x^2=2k^2-7<2k^2$ we get $|x|<k\sqrt2$, whence
\[
k'=3k-2x\geq 3k-2|x|>3k-2\sqrt2\,k=(3-2\sqrt2)k>0 ,
\]
so that $k'$ is again positive. Furthermore $k'<k$ if and only if $2k<2x$, that is, if and only if $x>k$; and since $x^2=2k^2-7$, for $x>0$ this happens exactly when $k^2>7$, that is, when $k\geq3$.

For a positive solution with $k\geq6$, we also have $x'>0$: indeed,
\[
x^2=2k^2-7>\frac{16}{9}k^2
\]
is equivalent to $2k^2>63$, and hence $3x>4k$. Thus every positive solution with $k\geq6$ descends to another positive solution with a strictly smaller second coordinate. The only positive solutions with $k<6$ are $(1,2)$ and $(5,4)$, and the latter descends to $(-1,2)$. Reversing the descent, every positive solution is obtained from either $(1,2)$ or $(5,4)$ by iterations of $(x,k)\mapsto(3x+4k,2x+3k)$. These are precisely the two announced orbits.  
\end{proof}

Combining the last two results we obtain infinitely many counterexamples to Conjecture~\ref{conj:two}.

\begin{corollary}\label{cor:infinite-family}
Let $(x_j,k_j)_{j\geq0}$ be either of the two orbits of Lemma~\ref{lem:pell}, and set $m_j=(x_j-1)/2$ and $a_j=m_j(m_j+1)$. Then, for every $j$ with $m_j\geq1$, the generalized sun graph
\[
\Sun_6(a_j)=C_{6,1}(0,a_j,a_j,2a_j,a_j,a_j)
\]
is an integral unicyclic graph on $6a_j+6$ vertices whose unique cycle has length six, with spectrum
\[
\Bigl\{0^{[6a_j-4]},\ \pm m_j^{[2]},\ \pm(m_j+1)^{[2]},\ \pm k_j\Bigr\}.
\]
In particular, Conjecture~\ref{conj:two} fails for infinitely many graphs.
\end{corollary}

\begin{proof}
Immediate from Theorem~\ref{thm:integrality-criterion} and Lemma~\ref{lem:pell}, since $a_j=m_j(m_j+1)$ and $2m_j^2+2m_j+4=k_j^2$ by construction.  
\end{proof}

The value $m=0$, which corresponds to the seed $(1,2)$ and to $a=0$, has been excluded from the statements only for simplicity: it describes the hexagon $C_6$ itself, which is of course integral, with spectrum $\{\pm2,\pm1^{[2]}\}$. The smallest genuine member of the family is the graph of Proposition~\ref{prop:counterexample}, obtained for $m=2$; Table~\ref{tab:family} lists the first few.

\begin{table}[ht]
\centering
\caption{The first integral generalized sun graphs $\Sun_6(a)=C_{6,1}(0,a,a,2a,a,a)$, where $a=m(m+1)$.}
\label{tab:family}
\begin{tabular}{@{}rrrrrl@{}}
\toprule
$m$ & $(x,k)$ & $a$ & $2a$ & order & spectrum\\
\midrule
$2$ & $(5,4)$ & $6$ & $12$ & $42$ & $\{0^{[32]},\pm2^{[2]},\pm3^{[2]},\pm4\}$\\
$5$ & $(11,8)$ & $30$ & $60$ & $186$ & $\{0^{[176]},\pm5^{[2]},\pm6^{[2]},\pm8\}$\\
$15$ & $(31,22)$ & $240$ & $480$ & $1446$ & $\{0^{[1436]},\pm15^{[2]},\pm16^{[2]},\pm22\}$\\
$32$ & $(65,46)$ & $1056$ & $2112$ & $6342$ & $\{0^{[6332]},\pm32^{[2]},\pm33^{[2]},\pm46\}$\\
$90$ & $(181,128)$ & $8190$ & $16380$ & $49146$ & $\{0^{[49136]},\pm90^{[2]},\pm91^{[2]},\pm128\}$\\
$189$ & $(379,268)$ & $35910$ & $71820$ & $215466$ & $\{0^{[215456]},\pm189^{[2]},\pm190^{[2]},\pm268\}$\\
\bottomrule
\end{tabular}
\end{table}

\subsection{The same shape on an arbitrary even cycle}\label{sec:general}

Nothing in the definition of $\Sun_6(a)$ is special to the hexagon. The
computation of Theorem~\ref{thm:char-poly-hexagon} relies only on the facts
that the cycle has even length, that two antipodal vertices are unbalanced, and
that all remaining vertices carry the same number of pendant vertices. It is
therefore natural to ask on which even cycles this shape produces integral
graphs. The answer, which is the main result of this section, is that only the
hexagon and the square do, so that no cycle of length $8,12,16,\dots$ occurs
even though all of these are multiples of four.

\begin{definition}\label{def:family}
Let $h\geq2$ and $a\geq1$ be integers. We denote by $\Sun_{2h}(a)$ the generalized sun graph
\[
\Sun_{2h}(a):=C_{2h,1}(n_1,\ldots,n_{2h}),
\qquad
n_1=0,
\quad
n_{h+1}=2a,
\quad
n_k=a \ \text{ for } k\notin\{1,h+1\},
\]
obtained from the cycle $C_{2h}=v_1v_2\cdots v_{2h}v_1$ by attaching no pendant vertex to $v_1$, exactly $2a$ pendant vertices to the antipodal vertex $v_{h+1}$, and exactly $a$ pendant vertices to each of the remaining $2h-2$ cycle vertices.
\end{definition}

For $h=2$ this is the family $C_{4,1}(0,a,2a,a)$ considered by Veras \cite{veras}, and for $h=3$ it is the family of Section~\ref{sec:hexagon}. In general, $\Sun_{2h}(a)$ is a bipartite unicyclic graph of order
\[
n=2h+\bigl((2h-2)a+2a\bigr)=2h(a+1),
\]
with maximum degree $2a+2$. The polynomials that govern its spectrum are the following analogues of the Chebyshev polynomials of the second kind, which arise as the continuants of a tridiagonal matrix with constant diagonal $\lambda^2-a$ and constant off-diagonal $-\lambda$; for $h=3$ we shall recover the factor $(\lambda^2-a)^2-\lambda^2$ met in Theorem~\ref{thm:char-poly-hexagon}.

\begin{definition}\label{def:U}
For a fixed integer $a\geq1$, define a sequence of polynomials $U_j=U_j(\lambda)$ by
\[
U_0=1,
\qquad
U_1=\lambda^2-a,
\qquad
U_j=(\lambda^2-a)U_{j-1}-\lambda^2U_{j-2}
\quad (j\geq2).
\]
\end{definition}

The polynomial $U_j$ has degree $2j$, and its roots admit a closed expression, which we record now since it is what will eventually decide integrality.

\begin{lemma}\label{lem:U-factorization}
For every integer $j\geq1$,
\[
U_j(\lambda)=\prod_{k=1}^{j}\left(\lambda^2-2\cos\frac{k\pi}{j+1}\,\lambda-a\right).
\]
\end{lemma}

\begin{proof}
Let $T_j$ denote the $j\times j$ tridiagonal matrix with all diagonal entries equal to $\lambda^2-a$ and all off-diagonal entries equal to $-\lambda$. Expanding $\det T_j$ along its last row gives exactly the recurrence of Definition~\ref{def:U}, with the same initial values, so $U_j=\det T_j$. On the other hand, $T_j=(\lambda^2-a)I_j-\lambda A(P_j)$, and the eigenvalues of the path $P_j$ are $2\cos\frac{k\pi}{j+1}$ for $1\le k\le j$. Hence
\[
U_j=\det\bigl((\lambda^2-a)I_j-\lambda A(P_j)\bigr)
=\prod_{k=1}^{j}\left(\lambda^2-a-2\cos\frac{k\pi}{j+1}\,\lambda\right),
\]
which is the assertion. 
\end{proof}

We can now compute the characteristic polynomial of $\Sun_{2h}(a)$ for every $h$ and every $a$. The proof follows the pattern of Theorem~\ref{thm:char-poly-hexagon}: the reflection symmetry splits the determinant of Lemma~\ref{lem:reduction} into two tridiagonal blocks. What is new is that the larger block can no longer be expanded by hand, and the unbalanced choice of $n_1=0$ and $n_{h+1}=2a$ is exactly what makes it collapse onto the smaller one.

\begin{theorem}\label{thm:general-char-poly}
Let $h\geq2$ and $a\geq1$ be integers, and let $n=2h(a+1)$ be the order of $\Sun_{2h}(a)$. Then
\begin{equation}\label{eq:general-char-poly}
p_{\Sun_{2h}(a)}(\lambda)
=
\lambda^{\,2h(a-1)+2}\,\bigl(\lambda^2-2a-4\bigr)\,U_{h-1}(\lambda)^2 .
\end{equation}
\end{theorem}

\begin{proof}
Write $b=2h$, $C=A(C_b)$ and $D=\diag(0,a,\dots,a,2a,a,\dots,a)$ as in Definition~\ref{def:family}, and set
\[
Q:=\lambda^2I_b-\lambda C-D,
\qquad
z:=\lambda^2 .
\]
By Lemma~\ref{lem:reduction} we have $p_{\Sun_{2h}(a)}(\lambda)=\lambda^{\,n-2b}\det Q$ with $n-2b=2h(a+1)-4h=2h(a-1)$, so it suffices to prove that
\begin{equation}\label{eq:detQ}
\det Q=z\,(z-2a-4)\,U_{h-1}^2 .
\end{equation}

The reflection of the cycle which fixes $v_1$ and $v_{h+1}$ and exchanges $v_{1+i}\leftrightarrow v_{1-i}$ for every $i$ preserves both $C$ and $D$, because the entries of $D$ were chosen symmetric with respect to it. Hence $Q$ is block diagonalized by the orthonormal basis of $\R^{b}$ formed by the $h-1$ antisymmetric vectors
\[
\frac{e_{1+i}-e_{1-i}}{\sqrt2},
\qquad 1\le i\le h-1
\]
(indices being read modulo $b$), together with the $h+1$ symmetric vectors
\[
e_1,\qquad
\frac{e_{1+i}+e_{1-i}}{\sqrt2}\ \ (1\le i\le h-1),
\qquad
e_{h+1} .
\]
On the antisymmetric subspace, the two vertices $v_1$ and $v_{h+1}$ contribute nothing, all diagonal entries equal $z-a$ and all off-diagonal entries equal $-\lambda$; the corresponding block is therefore the matrix $T_{h-1}$ of the proof of Lemma~\ref{lem:U-factorization}, so that its determinant is $U_{h-1}$. On the symmetric subspace one obtains the $(h+1)\times(h+1)$ tridiagonal matrix
\[
K_+=
\begin{pmatrix}
z & -\sqrt2\lambda & & & \\
-\sqrt2\lambda & z-a & -\lambda & & \\
 & \ddots & \ddots & \ddots & \\
 & & -\lambda & z-a & -\sqrt2\lambda\\
 & & & -\sqrt2\lambda & z-2a
\end{pmatrix},
\]
the two entries $\pm\sqrt2$ appearing because $e_1$ and $e_{h+1}$ are each adjacent to a symmetrized pair of cycle vertices. Thus
\begin{equation}\label{eq:detQ-split}
\det Q=U_{h-1}\cdot\det K_+ ,
\end{equation}
and \eqref{eq:detQ} amounts to the identity $\det K_+=z(z-2a-4)U_{h-1}$.

Let $\Delta_j$ denote the $j$-th leading principal minor of $K_+$. Expanding along the last row we get $\Delta_1=z$, then $\Delta_2=(z-a)z-2\lambda^2=z(z-a-2)$, and
\begin{equation}\label{eq:Delta-rec}
\Delta_j=(z-a)\Delta_{j-1}-z\,\Delta_{j-2}
\qquad (3\le j\le h),
\end{equation}
since $\lambda^2=z$, while the last step involves the two modified entries and reads
\begin{equation}\label{eq:Delta-last}
\Delta_{h+1}=(z-2a)\Delta_{h}-2z\,\Delta_{h-1} .
\end{equation}
The explicit expression
\begin{equation}\label{eq:Delta-explicit}
\Delta_j=2U_j+(2a-z)U_{j-1}
\qquad (1\le j\le h).
\end{equation}
follows directly from the recurrence. Indeed, its right-hand side satisfies \eqref{eq:Delta-rec}, and for $j=1,2$ it equals respectively
\[
2(z-a)+(2a-z)=z
\]
and
\[
2\bigl[(z-a)^2-z\bigr]+(2a-z)(z-a)=z(z-a-2),
\]
which are $\Delta_1$ and $\Delta_2$.

We finally insert \eqref{eq:Delta-explicit} into \eqref{eq:Delta-last}. This gives
\[
\Delta_{h+1}
=(z-2a)\bigl[2U_h+(2a-z)U_{h-1}\bigr]-2z\bigl[2U_{h-1}+(2a-z)U_{h-2}\bigr],
\]
that is,
\[
\Delta_{h+1}
=2(z-2a)U_h-(z-2a)^2U_{h-1}-4zU_{h-1}+2(z-2a)\,zU_{h-2} .
\]
The recurrence of Definition~\ref{def:U} in the form $zU_{h-2}=(z-a)U_{h-1}-U_h$ turns the last summand into $2(z-2a)\bigl[(z-a)U_{h-1}-U_h\bigr]$, so the two multiples of $U_h$ cancel and we obtain
\[
\Delta_{h+1}
=U_{h-1}\Bigl[-(z-2a)^2-4z+2(z-2a)(z-a)\Bigr]
=U_{h-1}\bigl[z^2-2az-4z\bigr]
=z\,(z-2a-4)\,U_{h-1}.
\]
Combining this with \eqref{eq:detQ-split} proves \eqref{eq:detQ}, and hence the theorem. 
\end{proof}

For $h=3$ we have $U_2=(\lambda^2-a)^2-\lambda^2$, and \eqref{eq:general-char-poly} reduces to \eqref{eq:char-poly-hexagon}, as it must. In general, formula \eqref{eq:general-char-poly} makes the integrality of $\Sun_{2h}(a)$ transparent: apart from the eigenvalue $0$, the spectrum consists of the two roots of $\lambda^2-2a-4$ together with the roots of $U_{h-1}$, each of the latter counted twice. The first factor imposes one Diophantine condition, which for $h=3$ was the Pell equation \eqref{eq:pell}; the second one, by Lemma~\ref{lem:U-factorization}, imposes conditions of a completely different nature, which turn out to be so restrictive that they leave only two possible cycle lengths.

\begin{theorem}\label{thm:only-4-and-6}
Let $h\geq2$ and $a\geq1$ be integers. If $\Sun_{2h}(a)$ is integral, then $h\in\{2,3\}$, that is, the cycle of $\Sun_{2h}(a)$ is either a square or a hexagon.
\end{theorem}

\begin{proof}
Suppose that $\Sun_{2h}(a)$ is integral. By Theorem~\ref{thm:general-char-poly}, every root of $U_{h-1}$ is then an integer. By Lemma~\ref{lem:U-factorization},
\[
U_{h-1}(\lambda)=\prod_{k=1}^{h-1}\left(\lambda^2-2\cos\frac{k\pi}{h}\,\lambda-a\right),
\]
and the roots of the $k$-th factor add up to $2\cos\frac{k\pi}{h}$. Since all these roots are integers, we conclude that
\[
2\cos\frac{k\pi}{h}\in\Z
\qquad\text{for every } 1\le k\le h-1 .
\]
Taking $k=1$, if $h\geq4$ then
\[
1<2\cos\frac{\pi}{h}<2,
\]
which is impossible for an integer. Hence $h\leq3$, and the assumption $h\geq2$ gives $h\in\{2,3\}$.   
\end{proof}

We stress that Theorem~\ref{thm:only-4-and-6} concerns the specific shape of Definition~\ref{def:family}, and not all generalized sun graphs; the exhaustive search reported in Section~\ref{sec:search}, however, found no integral graph of the form $C_{b,1}(n_1,\dots,n_b)$ with $b\geq7$ at all within the parameter ranges of Section~\ref{sec:search}. It remains to examine the case $h=2$, which does occur and produces a second infinite family.

For $h=2$ we have $\Sun_4(a)=C_{4,1}(0,a,2a,a)$, of order $4a+4$, and $U_1=\lambda^2-a$, so that \eqref{eq:general-char-poly} reads
\begin{equation}\label{eq:char-poly-square}
p_{\Sun_4(a)}(\lambda)=\lambda^{4a-2}\,(\lambda^2-a)^2\,(\lambda^2-2a-4).
\end{equation}
This family does not contradict Conjecture~\ref{conj:two}, but it is of independent interest. The three infinite families constructed in \cite{BRAGA2021281} have copies of $P_2$ attached, whereas here only pendant vertices are used. On the other hand, \cite[Corollary 4.3]{BRAGA2021281}, with the help of a result of de Lima, Mohammadian and Oliveira \cite{deLima2020}, identified one integral graph on $13$ vertices among the graphs $C_{4,1}(p,0,q,0)$, in which pendant vertices are attached to two nonadjacent cycle vertices only. Allowing a third cycle vertex to carry pendant vertices produces the infinite family above, which was first reported in \cite{veras}. For the sake of completeness, we provide a proof of the following result.

\begin{theorem}\label{thm:square-criterion}
Let $a\geq1$ be an integer. The generalized sun graph $\Sun_4(a)=C_{4,1}(0,a,2a,a)$ is integral if and only if $a$ and $2a+4$ are both perfect squares, say $a=c^2$ and $2a+4=k^2$, in which case it has $4a+4$ vertices and
\[
\sigma\bigl(\Sun_4(a)\bigr)=\Bigl\{0^{[4a-2]},\ \pm c^{[2]},\ \pm k\Bigr\}.
\]
Moreover, writing $c=2c'$ and $k=2k'$, the condition $2a+4=k^2$ is equivalent to the classical Pell equation
\begin{equation}\label{eq:pell-square}
k'^2-2c'^2=1 ,
\end{equation}
which has infinitely many solutions. Consequently there are infinitely many integral graphs of the form $\Sun_4(a)$.
\end{theorem}

\begin{proof}
By \eqref{eq:char-poly-square} the nonzero eigenvalues are $\pm\sqrt a$, each with multiplicity two, and $\pm\sqrt{2a+4}$; they are integers exactly when $a$ and $2a+4$ are perfect squares, and the multiplicities and the order are read off from \eqref{eq:char-poly-square}.

Suppose $a=c^2$ and $2c^2+4=k^2$. Then $k^2$ is even, hence $k$ is even, say $k=2k'$, and $4k'^2=2c^2+4$ gives $c^2=2k'^2-2$, so that $c$ is even as well, say $c=2c'$. Substituting, $4c'^2=2k'^2-2$, that is, $k'^2-2c'^2=1$, which is \eqref{eq:pell-square}; the steps are reversible. Finally, \eqref{eq:pell-square} is the classical Pell equation associated with $\sqrt2$, whose positive solutions are $(k'_j,c'_j)$ defined by $k'_j+c'_j\sqrt2=(3+2\sqrt2)^j$, $j\geq0$, and are therefore infinite in number.   
\end{proof}

The first solutions of \eqref{eq:pell-square} are $(k',c')=(1,0),(3,2),(17,12),(99,70),(577,408)$, which give $a=0,16,576,19600,$ \\ $665856$ and hence graphs of orders $4,68,2308,78404$ and $2663428$. The first of them is the square $C_4$ itself; the second one, obtained for $a=16$, is the smallest genuine member of the family. It is the graph $C_{4,1}(0,16,32,16)$ on $68$ vertices, and \eqref{eq:char-poly-square} gives
\[
p_{\Sun_4(16)}(\lambda)=\lambda^{62}(\lambda^2-16)^2(\lambda^2-36),
\qquad\text{that is,}\qquad
\sigma(\Sun_4(16))=\bigl\{0^{[62]},\pm4^{[2]},\pm6\bigr\}.
\]

\section{Computational results}\label{sec:search}

Lemma~\ref{lem:reduction} provides an efficient way to test integrality for the graphs $C_{b,1}(n_1,\dots,n_b)$. Put $D=\diag(n_1,\dots,n_b)$ and $C=A(C_b)$. The roots of
\[
\det(\lambda^2I_b-\lambda C-D)
\]
are the eigenvalues of the symmetric $2b\times2b$ matrix
\begin{equation}\label{eq:symmetric-linearization}
\mathcal H(n_1,\dots,n_b)
=
\begin{pmatrix}
C&D^{1/2}\\
D^{1/2}&0
\end{pmatrix}.
\end{equation}
Thus the nonzero eigenvalues of $C_{b,1}(n_1,\dots,n_b)$ can be tested using a matrix whose order depends only on $b$, rather than on the number of pendant vertices.

We enumerated the parameter vectors up to the dihedral action on the cycle. Numerical eigenvalues of \eqref{eq:symmetric-linearization} were used only as a screening step; every candidate declared integral was then certified by computing and factoring $\det(\lambda^2I_b-\lambda C-D)$ exactly over $\Z[\lambda]$. The search covered all vectors, not identically zero, in the following boxes:
\[
\begin{aligned}
&b=3,\ n_k\leq40; \qquad
b=4,\ n_k\leq34; \qquad
b=5,\ n_k\leq30; \qquad
b=6,\ n_k\leq14;\\
&b=7,\ n_k\leq10; \qquad
b=8,\ n_k\leq9; \qquad
b=9,10,\ n_k\leq4 .
\end{aligned}
\]

For $b=4$ the search produced exactly two integral graphs: $C_{4,1}(6,0,3,0)$ on $13$ vertices, which appears in the first row of \cite[Table~1]{BRAGA2021281}, and $C_{4,1}(0,16,32,16)$ on $68$ vertices, the smallest nontrivial member of Theorem~\ref{thm:square-criterion}. For $b\notin4\Z$, the only integral graph found was $C_{6,1}(0,6,6,12,6,6)$, the $42$-vertex counterexample of Proposition~\ref{prop:counterexample}. In particular, no integral graph was found for $b\in\{3,5,7,8,9,10\}$, and for $b=6$ the counterexample is unique within the box $0\leq n_k\leq14$.

These statements are exhaustive only within the displayed coordinate-wise bounds. In particular, the search does not prove that the graph on $42$ vertices has minimum order among all counterexamples to Conjecture~\ref{conj:two}: a parameter vector with some $n_k>14$ can still define a graph of order below $42$, and generalized sun graphs containing copies of $P_2$ or both types of pendant paths are not included. What the theoretical results do prove is that the $42$-vertex graph is the smallest nontrivial member of the family $\Sun_6(a)$.

\section{Conclusions and discussion}\label{sec:conclusion}

The results presented in this paper reveal two different mechanisms governing the integrality of graphs. The first one is local: integrality is sensitive to a single pendant path of three edges and no feature of the remaining graph can compensate it. The second one is global: once long pendant paths are absent, integrality can occur but with a highly structured global configuration whose existence is controlled by Diophantine conditions. The two conjectures of Braga, Del-Vecchio and Rodrigues \cite{BRAGA2021281} illustrate this contrast well. The first one holds in a considerably stronger form than originally stated, for arbitrary simple graphs and in fact for arbitrary real symmetric matrices. The second one fails, and not sporadically but along an infinite family.

The local mechanism is the content of Theorem~\ref{thm:long-pendant-paths}. A
pendant path with at least three edges forces eigenvalues into the intervals
$(1,\rho]$ and $[-\rho,-1)$, no matter what the rest of the graph looks like: the proof uses no property of that part beyond real symmetry, and it applies to an arbitrary real symmetric matrix to which a terminal three-coordinate path is adjoined. The theorem reduces the study of integral generalized sun graphs to pendant paths of length one or two, that is, to copies of $P_1$ and $P_2$. These two types behave in opposite ways, and the eigenvalue $0$ is what
separates them. Let $G$ be a generalized sun graph, other than a cycle, whose
cycle is $C_b$. Copies of $P_2$ do not create the eigenvalue $0$: deleting a
pendant vertex together with its neighbour preserves the nullity, so if all the attachments of $G$ are copies of $P_2$, then $G$ and $C_b$ have the same nullity. If moreover $G$ is integral, then $G$ has the eigenvalue $0$, since an
integral unicyclic graph with no eigenvalue $0$ is $C_3$ or $C_6$
\cite[Theorem~1]{Omidi2009}; hence so does $C_b$, and therefore $b$ is a multiple of four, for $0\in\sigma(C_b)\iff b\equiv 0\pmod 4$. All the infinite families known before carried copies of $P_2$, which is why the evidence accumulated at cycle length four. Pendant vertices, on the contrary, create the eigenvalue $0$ in abundance: $n_k$ pendant vertices at $v_k$ contribute $n_k-1$ zero eigenvalues, so already two of them at a common cycle vertex force it. These
zero eigenvalues come from the attachments alone and carry no information about
the cycle, so the argument above breaks down and no constraint on the cycle
length survives. The graph $C_{6,1}(0,6,6,12,6,6)$ is a striking illustration,
with $32$ zero eigenvalues although $C_6$ has none. 

What makes this second mechanism tractable is Lemma~\ref{lem:reduction}, which
compresses the characteristic polynomial of $C_{b,1}(n_1,\ldots,n_b)$ into a
determinant of order $b$, regardless of the number of pendant vertices. This is
how the graph above was found. It is also what shows that the example is not
isolated: the family $\Sun_6(a)=C_{6,1}(0,a,a,2a,a,a)$ is integral exactly for the parameters described in Theorem~\ref{thm:integrality-criterion}, which reduce to the generalized Pell equation \eqref{eq:pell}. The failure of Conjecture~\ref{conj:two} is therefore an arithmetic phenomenon
rather than an isolated coincidence. The same shape makes sense on any even
cycle, and Theorem~\ref{thm:only-4-and-6} shows that within it integrality
forces the cycle to be a square or a hexagon: no cycle of length
$8,12,16,\dots$ occurs, although all of these are multiples of four. The two
surviving cases give the hexagonal family, whose first members are the graphs on $42$, $186$ and $1446$ vertices of Table~\ref{tab:family}, and the square family $\Sun_4(a)=C_{4,1}(0,a,2a,a)$, whose first member other than $C_4$ is the graph $C_{4,1}(0,16,32,16)$ on $68$ vertices. Both are infinite and both are
governed by Pell equations. They are not, however, a classification: even for
these two cycle lengths, other integral graphs exist, as the $13$-vertex graph
$C_{4,1}(6,0,3,0)$ shows on the square.

The cycle lengths that actually occur are still unknown: the present paper
shows that $4$ and $6$ do, and it is natural to ask whether $b$ must be even
and whether these are the only possibilities. A more tractable version
restricts attention to pendant vertices, where the problem is to classify the
integral graphs $C_{b,1}(n_1,\ldots,n_b)$ up to the dihedral action of the
cycle. Combining the two local types left available by
Theorem~\ref{thm:long-pendant-paths} is where our arguments stop: when both are
present, the nullity of the graph is no longer controlled by that of its cycle,
since a single pendant vertex may lower it while two at a common vertex raise
it, so the eigenvalue $0$ carries no information about the cycle length. Integral examples of this mixed type do exist, such as the $20$-vertex graph $C_4(5P_2,P_1,2P_2,P_1)$ found in \cite{BRAGA2021281}, and classifying them seems to us a natural next step. Beyond generalized sun graphs, the matrix argument of Section~\ref{sec:conj1}
suggests a question of a different nature: the terminal path $P_3$ is a rooted
graph whose attachment to an arbitrary symmetric matrix through one coordinate
already forces a nonintegral eigenvalue, and one would like to characterize the
rooted graphs with this universal property, together with the intervals
avoiding the integers that their spectra are forced to meet. This points to a
general strategy for the study of integral graphs: first identify the small
rooted configurations that are spectrally forbidden, then analyse the remaining
attachment patterns through low-dimensional matrix polynomials and the
Diophantine restrictions imposed by their factorizations.

\bigskip

\noindent\textbf{Declaration of competing interest:}
The authors declare that there is no competing interest.

\bigskip
\noindent\textbf{Acknowledgements}
Rodrigo O. Braga acknowledges the support of CNPq grant 408180/2023-4.


\end{document}